\documentclass[runningheads]{llncs}
\usepackage[T1]{fontenc}
\usepackage{graphicx}
\usepackage{amsmath, amssymb,  mathabx, amsfonts,enumerate}
\usepackage[all]{xy}
\usepackage{amscd,stmaryrd}
\usepackage[all]{xy}
\usepackage{comment}
\usepackage{mathtools}
\usepackage{tikz} 
\usepackage{tikz-cd} 
\tikzcdset{diagrams={nodes={inner sep=7.5pt}}}

\usepackage{url}
\usepackage{hyperref}
 
\newtheorem{theoremletter}{Theorem}

\numberwithin{equation}{section}
 
\newcommand {\N}{\mathbb{N}} 
\newcommand {\Z}{\mathbb{Z}}

\newcommand{\GG}{\mathcal{G}}

\newcommand{\LL}{\mathcal{L}}

\DeclareMathOperator{\End}{End}

\DeclareMathOperator{\Id}{Id}

\begin{document}

\title
{On Gottschalk's surjunctivity conjecture for non-uniform cellular automata}   
%
\author{Xuan Kien Phung\inst{1,2}\orcidID{0000-0002-4347-8931}}
\authorrunning{Phung, X.K.}
%
\institute{Département de Mathématiques et de Statistique, Université de Montréal, Montréal, Québec, H3T 1J4, Canada 
\and Département d'informatique et de recherche opérationnelle,  Université de Montréal, Montréal, Québec, H3T 1J4, Canada\\ 
\email{phungxuankien1@gmail.com \\    xuan.kien.phung@umontreal.ca}} 
\maketitle              
\begin{abstract}
Gottschalk's surjunctivity conjecture for a group  $G$ states  that it is impossible for cellular automata (CA) over the universe $G$ with finite alphabet to produce strict embeddings of the full shift into itself. A group universe $G$ satisfying Gottschalk's surjunctivity conjecture is called a surjunctive group. The surjunctivity theorem of Gromov and Weiss shows that every sofic group is surjunctive. In this paper, we study the  surjunctivity of local perturbations of CA and more generally of non-uniform cellular automata (NUCA) with finite memory and  uniformly bounded singularity over surjunctive  group universes. In particular, we show that such a NUCA must be invertible whenever it is reversible.  We also obtain similar results which extend to the class of NUCA a certain dual surjunctivity theorem of Capobianco, Kari, and Taati for CA.

\end{abstract}

\section{Introduction}  
In computational science and engineering, cellular automata (CA), especially reversible CA,  arise as a fundamental and powerful model of simulation for various   physical and biological systems \cite{wolfram-book} whose global evolution is described by spatially uniform local transition rules. In computer science, it is well-known that the CA called Game of Life of Conway \cite{GOL} is Turing complete. Some  notable recent mathematical theory of CA achieves a  dynamical characterization of amenable groups by the Garden of Eden theorem for CA (pre-injectivity$\iff$surjectivity, see Section~\ref{s:definitions} below)  \cite{moore}, \cite{myhill}, \cite{ceccherini}, \cite{bartholdi-kielak} and establishes the equivalence between Kaplansky's stable finiteness conjecture \cite{kap} for group rings and  Gottschalk's surjunctivity conjecture (injectivity$\implies$surjectivity)  \cite{gottschalk} for linear CA  \cite{phung-geometric}, \cite{phung-weakly}, \cite{phung-cjm}, \cite{cscp-model}. In this paper, we first explore Gottschalk's surjunctivity conjecture (see Section~\ref{s:intro-gottschalk}) over surjunctive group universes (see Definition~\ref{def:surjunctive}), e.g. sofic groups such as amenable or residually finite groups, for non-uniform CA (NUCA) with finite memory. We then obtain an extension of a well-known dual version of  Gottschalk's surjunctivity conjecture due to Capobianco, Kari, and Taati \cite{kari-post-surjective} (see Section~\ref{s:intro-gottschalk}) over the class of post-injunctive groups (see Definition~\ref{def:post-injunctive}), e.g. sofic groups, also for NUCA with finite memory.  Our main results cover NUCA which are local perturbations of CA in which a finite number of cells can follow   local transition rules different from the local transition rule of the underlying CA. Among applications, such NUCA have  close connections with  asynchronous CA. For example,  if $F$ is an asynchronous CA in which only a finite number of cells are allowed to update  at each (discrete) time, e.g. fully asynchronous CA \cite{full-asyn-ca} and skew-asynchronous CA \cite{skew-asyn-ca}, then $F$ can be identified with a sequence $(F_n)_{n \in \N}$ of local perturbations of the identity CA: the system $F$ at time $n\in \N$ is exactly the NUCA $F_n \circ F_{n-1} \circ\dots \circ  F_0$. 
\par 
We generalize our results to a certain class of global perturbations of CA with uniformly bounded singularity. Essentially by definition and basic properties of CA, every reversible CA with finite memory over a surjunctive group universe must be invertible (see Section~\ref{s:definitions} for the exact definitions).   Gottschalk’s surjunctivity  conjecture amounts to saying that  every group is surjunctive. 
Hence, our results extend Gottschalk’s surjunctivity conjecture by showing that over surjunctive group universes, reversible NUCA with finite memory and uniformly bounded singularity must be  invertible (see Theorem~\ref{t:main-A} and Theorem~\ref{t:main-B}). 
\par  
\subsection{Basic definitions}
\label{s:definitions}
We recall notions of symbolic dynamics. 
Given a discrete set $A$ and a group $G$, a \emph{configuration} $x \in A^G$ is a map $x \colon G \to A$.  Two configurations $x,y  \in A^G$ are \emph{asymptotic} if  $x\vert_{G \setminus E}=y\vert_{G \setminus E}$ for some finite subset $E \subset G$.  We say that $x\in A^G$ is \emph{asymptotically constant} if $x$ is asymptotic to some constant configuration $y \in A^G$. The \emph{Bernoulli shift} action $G \times A^G \to A^G$ is defined by $(g,x) \mapsto g x$, 
where $(gx)(h) =  x(g^{-1}h)$ for  $g,h \in G$,  $x \in A^G$. The \emph{full shift} $A^G$ is equipped with the prodiscrete topology. For each $x \in A^G$, we define $\Sigma(x) = \overline{\{gx \colon g \in G\}} \subset A^G$ as the smallest closed subshift containing~$x$. For every subset $X\subset A^G$, we denote the restriction of $X$ to a subset  $F\subset G$ by  $X_F=\{x\vert_F \colon \in X\}$. 
\par 
Generalizing the first construction of a CA over $\Z^2$ by von Neumann 
  \cite{neumann} (see also  \cite{burks}),
a CA over the group $G$ (the \emph{universe}) and the set $A$ (the \emph{alphabet}) is   a $G$-equivariant and uniformly continuous  self-map $A^G \righttoleftarrow$  \cite{hedlund-csc}, \cite{hedlund}.  One usually refers to each element $g \in G$ as a cell of the universe. Then every CA is uniform in the sense that all the cells follow the same local transition map. More generally, when different cells can evolve according to different local transition maps to break down the above  uniformity, we obtain NUCA  \cite{Den-12a}, \cite{Den-12b}, \cite[Definition~1.1]{phung-tcs}:    

\begin{definition}
\label{d:most-general-def-asyn-ca}
Given a group $G$ and an alphabet $A$,  let $M \subset G$ and let $S = A^{A^M}$ be the set of all maps $A^M \to A$. For  $s \in S^G$, we define the NUCA $\sigma_s \colon A^G \to A^G$  by  
$
\sigma_s(x)(g)=  
    s(g)((g^{-1}x)  
	\vert_M)$ for all $x \in A^G$ and $g \in G$. 
 \end{definition} 
\par 
The set $M$ is called a \emph{memory} and $s \in S^G$  the \textit{configuration of local transition maps or local defining maps} of $\sigma_s$.  Every CA is thus a NUCA with finite memory and constant configuration of local defining maps. Conversely, we regard NUCA as \emph{global perturbations} of CA.  When $s, t \in S^G$ are asymptotic,  $\sigma_s$ and $\sigma_t$ are mutually \emph{local perturbations} of each other. We also observe that 
 NUCA with finite memory as defined in this paper form a proper subset of the set of  uniformly continuous selfmaps of the full shift (see also \cite{salo-2}[Proposition~14]), e.g. the map $F\colon A^G\to A^G$ given by $F(x)(g)=x(1_G)$ for all $x \in A^G$ and $g\in G$ is uniformly continuous but it is not a NUCA with finite memory.  
 As for CA,  NUCA with finite memory satisfy the closed image property \cite[Theorem~4.4]{phung-tcs}, several decidable and undecidable properties \cite{Den-12a}, \cite{kari-20}, \cite{phung-decidable-nuca} 
and  variants of the Garden of Eden theorem \cite{paturi}, \cite{phung-GOE-NUCA}. 
\par 
In this paper we analyze  the relations between the following  dynamic properties of $\sigma_s$. 
We say that   $\sigma_s$ is \emph{pre-injective} if $\sigma_s(x) = \sigma_s(y)$ implies $x= y$ whenever $x, y \in A^G$ are asymptotic. Similarly, $\sigma_s$ is \emph{post-surjective} if for all $x, y \in A^G$ with $y$ asymptotic to $\sigma_s(x)$,  then   $y= \sigma_s(z)$ for some $z \in A^G$ asymptotic to $x$. By the closed image property, it is known that post-surjectivity implies surjectivity for NUCA with finite memory. We say that $\sigma_s$ is  \emph{stably injective}, resp. \emph{stably post-surjective},  if  $\sigma_p$ is injective, resp. post-surjective,   for every $p \in \Sigma(s)$. 
\par 
The NUCA $\sigma_s$ is said to be \emph{reversible} or \emph{left-invertible} if there exists a NUCA with finite memory $\tau \colon A^G \to A^G$ such that $\tau \circ \sigma_s = \Id_{A^G}$.  A NUCA with finite memory is reversible if and only if it is stably injective (see \cite[Theorem~A]{phung-tcs}).  We say that  $\sigma_s$ is  \emph{invertible} if it is bijective and the inverse map  $\sigma_s^{-1}$ is a NUCA with {finite} memory \cite{phung-tcs}. We define also \emph{stable invertibility} which is in general stronger than invertibility, namely, $\sigma_s$ is stably invertible if there exist $N\subset G$ finite and $t \in T^G$ where $T=A^{A^N}$ such that for every $p \in \Sigma(s)$, we have   $\sigma_p \circ \sigma_q=\sigma_q \circ \sigma_p=\Id$  for some $q \in \Sigma(t)$.  In fact, we will show (see  Lemma~\ref{l:equivalent-stable-invertible}) that every invertible NUCA with finite memory over a finite alphabet and a countable group universe is automatically stably invertible.

\begin{remark}
\label{r:remark}
For CA, note that stable invertibility, resp. stable injectivity, resp. stable post-surjectivity, is equivalent to invertibility, resp.   injectivity, resp. post-surjectivity since $\Sigma(s)=\{s\}$ if the configuration $s \in S^G$ is constant.
\end{remark}

\begin{remark}
Suppose that the universe $G$ is a countably  infinite group and $s\in S^G$ is  asymptotic to a constant configuration $c\in S^G$. By \cite{phung-tcs}[Lemma~8.1] (which holds more generally for a countably infinite group $G$ by a similar proof), we have  $\Sigma(s)=\{c\}\cup\{gs\colon g \in G\}$. Consequently, a dynamic property  is stable for $\sigma_s$ simply means that the NUCA $\sigma_s$ and the CA $\sigma_c$ satisfy the same property. 
\end{remark}

\subsection{Main results}
\label{s:intro-gottschalk}
 
Gottschalk's surjunctivity conjecture  \cite{gottschalk} asserts that over any group universe, every CA with finite alphabet must be \emph{surjunctive} (injectivity$\implies$surjectivity). In other words, it is impossible for CA to  produce strict embeddings of the full shift into itself.   

\begin{definition}
\label{def:surjunctive}
A group $G$ is said to be surjunctive if for every finite alphabet $A$, every injective CA $\tau\colon A^G\to A^G$ must be surjective. 
\end{definition}

Every CA with finite memory is injective if and only if it is  reversible. Therefore, a group $G$ is surjunctive if and only if for every finite alphabet $A$, every reversible CA $\tau\colon A^G\to A^G$ must be invertible. 
Over the wide class of sofic group universes, the surjunctivity conjecture was famously shown by Gromov and Weiss in  \cite{gromov-esav},  \cite{weiss-sgds}.

\begin{theorem}[Gromov-Weiss]
\label{t:g-w}
    Every sofic group is surjunctive. 
\end{theorem}

The class of sofic groups was first  introduced by Gromov  \cite{gromov-esav} and includes all amenable groups and all residually finite groups. The question of whether there exists a non-sofic group is still open. 
\par 
 The situation for the surjunctivity of NUCA admits some complications. While injective and even reversible NUCA with finite memory may fail to be  surjective (see \cite[Example 14.3]{phung-tcs}), results in \cite{phung-tcs} show that reversible local perturbations of CA with finite memory over an amenable group or a residually finite group universe must be surjective. 
More generally, 
we obtain an extension  (Theorem~\ref{t:main-A})  to cover all reversible, or equivalently {stably injective}, local perturbations of CA over surjunctive group universes. We also strengthen the conclusion by showing that such NUCA must be stably invertible (see Section~\ref{s:a}). 
 
\par 
\begin{theoremletter}
\label{t:main-A}
Let $M$  be a finite subset of a countable surjunctive group $G$. Let $A$ be a finite  alphabet and  $S=A^{A^M}$. Suppose that $\sigma_s$ is stably injective for some asymptotically constant $s \in S^G$. Then $\sigma_s$ is stably invertible.  
\end{theoremletter}
\par 
Combining with Theorem~\ref{t:g-w} and the result in \cite[Theorem~B]{phung-tcs} where we can replace the stable injectivity by the weaker injectivity condition whenever $G$ is an amenable group, we obtain the following general surjunctivity and invertibility result for NUCA which are local perturbations of  CA. 
\par 
\begin{corollary}
\label{cor:main-intro}
  Let $M$  be a finite subset of a countable group $G$. Let $A$ be a finite alphabet and  $S=A^{A^M}$. Let $s \in S^G$ be an asymptotically constant configuration. Then $\sigma_s$ is stably invertible in each of the following cases: 
  \begin{enumerate}[\rm (i)] 
      \item $G$ is an amenable group and $\sigma_s$ is injective;
      \item 
      $G$ is a sofic group and $\sigma_s$ is reversible. 
  \end{enumerate}
\end{corollary}
 
 \par 
Our next  results concern a certain dual surjunctivity  version of Gottschalk's conjecture introduced   by 
Capobianco, Kari, and Taati in \cite{kari-post-surjective} which  states that  every post-surjective CA over a group universe and a finite alphabet is also pre-injective. The authors settled in the same paper \cite{kari-post-surjective} the case of sofic group universes. 
See also \cite{phung-post-surjective} for some extensions.

\begin{theorem}[Capobianco-Kari-Taati] 
\label{t:dual}Let $G$ be a sofic group and let $A$ be a finite alphabet. Then every post-surjective CA $\tau\colon A^G\to A^G$ is pre-injective. 
\end{theorem}
\par 
The above result motivates the following notion of \emph{post-injunctive} groups. 
\begin{definition}
    \label{def:post-injunctive} A group $G$ is  {post-injunctive} if for every finite alphabet $A$, every post-surjective CA $\tau\colon A^G\to A^G$ must be pre-injective. 
\end{definition}

In particular, every sofic group is post-injunctive. By a similar technique as in the proof of Theorem~\ref{t:main-A}, we establish (see Section~\ref{s:b}) the following  extension of the above result of Capobianco, Kari, and Taati to cover the class of stably post-surjective local perturbations of CA over post-injunctive group universes.  

\begin{theoremletter}
\label{t:main-C} 
Let $G$ be a countable post-injunctive group and let $A$ be a finite alphabet.  Let $M\subset G$ be finite and $S=A^{A^M}$. Let $s \in S^G$ be  asymptotically constant such that $\sigma_s$ is stably post-surjective. Then $\sigma_s$ is stably invertible. 
\end{theoremletter}
\par 
Note that residually finite groups are sofic and thus post-injunctive. It follows that \cite[Theorem~D.(ii)]{phung-laa} is a consequence of our   
Theorem~\ref{t:main-C} when the group universe $G$ is residually finite.  
\par 
Our main results also hold for global perturbations of CA with \emph{uniformly bounded singularity}. A NUCA $\sigma_s \colon A^G \to A^G$ has {uniformly bounded singularity} if for every  finite subset $E\subset G$  with $1_G\in E=E^{-1}$, we can find a finite subset $F\subset G$ containing $E$  such that the restriction  $s\vert_{FE\setminus F}$ is constant. Note that in this paper, group operations are to be performed before boolean set operations. In particular, $FE\setminus F=(FE)\setminus F$. 
It is clear that $\sigma_s$ has uniformly bounded singularity if  $s$ is asymptotically constant. When the universe $G$ is a residually finite group, our notion of uniformly bounded singularity is closely related but not equivalent to the notion of (periodically) bounded singularity for NUCA defined in   \cite[Definition~10.1]{phung-tcs}. For  NUCA with uniformly
bounded singularity, we obtain the following generalization  of Theorem~\ref{t:main-A}  (see  Section~\ref{s:c}) 
 
\par 
\begin{theoremletter}
    \label{t:main-B} 
 Let $M$ be a finite subset of a finitely generated surjunctive group~$G$. Let $A$ be a finite alphabet and let $S=A^{A^M}$. Suppose that $\sigma_s \colon A^G \to A^G$ is stably injective for some $s \in S^G$ with uniformly bounded singularity. Then $\sigma_s$ is stably invertible.     
\end{theoremletter}

Similarly, we can extend  Theorem~\ref{t:main-C} to the context of   NUCA with uniformly
bounded singularity   (see Section~\ref{s:d}). 
\begin{theoremletter}
\label{t:main-D} 
 Let $G$ be a finitely generated post-injunctive group and let $A$ be a finite alphabet.  Let $M\subset G$ be finite and let $S=A^{A^M}$. Suppose that $\sigma_s \colon A^G~\to~A^G$ is stably post-surjective for some $s \in S^G$ with uniformly bounded singularity. Then $\sigma_s$ is stably invertible. 
\end{theoremletter}

\par 
\subsection{Perspectives} 
Our main results  motivate the following natural questions which have an affirmative answer for surjunctive group universes and post-injunctive group universes respectively both in the setting of NUCA with finite memory and uniformly bounded singularity.  

\begin{question}
\label{conjecture-stable-surj} Let $G$ be a   group universe and let $A$ be a finite alphabet.  Let $M\subset G$ be finite and let $S=A^{A^M}$. Suppose that $\sigma_s \colon A^G \to A^G$ is stably injective for some configuration $s \in S^G$. Is $\sigma_s$ stably invertible?
\end{question}

\begin{question}
\label{Stable dual-surjunctivity}
Let $G$ be a  group universe and let $A$ be a finite alphabet.  Let $M\subset G$ be finite and let $S=A^{A^M}$. Suppose that $\sigma_s \colon A^G \to A^G$ is stably post-surjective for some configuration $s \in S^G$. Is $\sigma_s$  stably invertible?  
\end{question}

\par 
Investigating the relation between these two questions, we can show (see Corollary~\ref{pro:intro}) that Question~\ref{conjecture-stable-surj} and  Question~\ref{Stable dual-surjunctivity}  are in fact  equivalent when restricted to the class of linear NUCA with finite memory. For every (possibly infinite) vector space $V$, note that  $V^G$ is  a vector space with component-wise operations and a NUCA $\tau \colon V^G \to V^G$ is said to be \emph{linear} if it is also a linear map of vector spaces or equivalently, if its local transition maps are   linear. Proposition~\ref{p:prop-intro-dual} below provides a more precise relation between  Question~\ref{conjecture-stable-surj} and Question~\ref{Stable dual-surjunctivity} than just their equivalence in the setting of linear NUCA with finite memory. For the proof of Proposition~\ref{p:prop-intro-dual},  we will need the following notion of dual NUCA introduced in \cite{phung-laa}. 
\par 
Let $M\subset G$ be a finite subset and let $A$ be a vector space. Let $s\in S^G$ where $S =  \LL(A^M, A)$  denotes the space of linear maps from $A^{M}$ to $A$. By linearity, we can define the linear endomorphism $s(g,m) \in \End(A)$ for all $m \in M$ and $g \in G$ by the formula 
\begin{align}
\label{e:dual-nuca}
s(g)(v) = \sum_{m \in M} s(g,m)v(m) \quad \text{ for all }v \in A^M.
\end{align}
Setting $s(g,m)=0$ for all  $m \in G \setminus M$, we can write for all $v \in A^G$ that 
$$
s(g)(v) = \sum_{m \in G} s(g,m)v(m).
$$
Let $A^*$ be the dual space of $A$ and let $T = \LL(A^{*M^{-1}}, A^*)$. We use the right superscript $^\mathsf{T}$ to denote the transpose of linear maps. The dual configuration of local defining maps $s^* \in T^G$  is  given by  
$ 
s^*(g,m)\coloneqq  
s(gm,m^{-1})^\mathsf{T}$  for all $g,m \in G$. In particular, $s^*(g,m)=0$ for all $g\in G$ whenever $m^{-1} \notin M$ or equivalently when $m\notin M^{-1}$. Therefore,   for all $u \in (A^*)^{M^{-1}}$ and $g \in G$, we have 
    \begin{align}
    \label{e:dual-formula}
    s^*(g)(u) =  \sum_{m \in M^{-1}} s^*(g,m) u(m) = \sum_{m \in M^{-1}} s(gm, m^{-1})^\mathsf{T} u(m). 
    \end{align}
    The linear NUCA $\sigma_s^*\coloneqq  \sigma_{s^*}$ is called  
 the \emph{dual linear NUCA} of $\sigma_s$.  We note that    $s^{**}=s$ and 
$\sigma_{s}^{**}=\sigma_s$ for all $s \in S^G$ (see  \cite[Lemma~5.2]{phung-laa}).

\begin{proposition}
\label{p:prop-intro-dual}
Let $M$ be a finite subset of a group $G$ and let $A$ be a finite vector space alphabet. Let $S =  \LL(A^M, A)$ and let $s\in S^G$. Then the following  equivalences hold true: 
\begin{enumerate}[\rm (i)]
    \item $\sigma_s$ is invertible$\iff$$\sigma_{s^*}$ is invertible,
    \item 
    $\sigma_s$ is stably injective$\iff$$\sigma_{s^*}$ is stably post surjective,
    \item 
     $\sigma_s$ is stably post-surjective$\iff$$\sigma_{s^*}$ is stably injective,
     \item 
 $\sigma_s$ has uniformly bounded singularity if and only if so is $\sigma_{s^*}$.
 \item 
  $s$ is asymptotically constant  if and only if so is $ {s^*}$.
\end{enumerate}
    
\end{proposition}

\begin{proof} 
The equivalences (i), (ii), and (iii) result directly from \cite[Theorem~A]{phung-laa}. For (iv), let $E\subset G$ be a finite set such that $1_G\in E=E^{-1}$. Let $K\subset G$ be any finite subset  large enough such that $1_G\in K = K^{-1}$ and $MEM^{-1}\subset K$. Suppose first that $\sigma_s$ has uniformly bounded singularity.  Then there exists a finite subset $H\subset G$ such that $s\vert_{HK\setminus H}$ is constant. Let $F= HM$. We claim that $s^*\vert_{FE\setminus F}$ is also constant. Indeed, since $s\vert_{HK\setminus H}$ is constant and 
$$
s(g)(v) = \sum_{m \in M} s(g,m)v(m) \quad \text{ for all } v \in A^M, \,  g\in G, 
$$
there exists $t(m) \in \mathcal{L}(A^M, A)$ for every $m \in M$ such that  
\begin{align}
\label{l:dual-proof-1}
s(g,m) = t(m) \quad \text{ for all } g \in HK \setminus H,\,    m \in M. 
\end{align}
Let $g \in FE\setminus F$ and $m \in M^{-1}$. In particular, $m^{-1}\in M$ and  $g\notin F=HM$.  Thus $g\notin Hm^{-1}$ and we deduce that $gm\notin H$. On the other hand, $m \in M^{-1}$ and $g \in FE=HME$ thus $gm \in HMEm \subset HMEM^{-1}$. As $MEM^{-1}\subset K$, it follows that $gm \in HK$. We conclude that $gm\in HK\setminus H$. We then infer from \eqref{l:dual-proof-1}  that  $s(gm, m^{-1}) = t(m^{-1})$  for all $g \in FE\setminus F$ and $m \in M^{-1}$. Consequently, the formula \eqref{e:dual-formula} implies that for all $u \in (A^*)^{M^{-1}}$, we have 
\begin{align*}
s^*(g)(u) = \sum_{m \in M^{-1}} s(gm, m^{-1})^\mathsf{T} u(m) = \sum_{m \in M^{-1}} t(m^{-1})^\mathsf{T} u(m)  \quad \text{ for all } g \in FE\setminus F. 
\end{align*} 
Hence, $s^*\vert_{FE\setminus F}$ is constant. This proves that $\sigma_{s^*}$ has uniformly bounded singularity whenever $\sigma_s$ does. Since $(s^*)^* =s$, the converse also holds and the proof of (iv) is complete. 
\par
For (v), suppose that $s$ is asymptotically constant, that is, $s\vert_{G\setminus H}$ is constant for some finite subset $H\subset G$. Then a similar (but simler) argument as above shows that $s^*\vert_{G\setminus HM}$ is constant. Since $HM$ is finite, this means that $s^*$ is also asymptotically constant whenever $s$ is  asymptotically constant. As $(s^*)^*=s$, the converse is true and we can conclude that $s$ is asymptotically constant if and only if so is $s^*$.  
\qed   
\end{proof}

\begin{corollary}
\label{pro:intro}
Let $G$ be a countable group and let $A$ be a finite vector space alphabet. Then  Question~\ref{conjecture-stable-surj} has an affirmative answer  for all  $\mathrm{NUCA}$ $\tau\colon A^G\to A^G$ in the class $C$ if and only if so does Question~\ref{Stable dual-surjunctivity} where $C$ is one of the following:    
\begin{enumerate}[\rm (a)]
\item $C =\{\text{linear } \mathrm{NUCA} \text{ with finite memory}\}$; 
\item $C = \{\text{linear } \mathrm{NUCA} \text{  which are local perturbations of a linear CA}\}$; 
    \item $C = \{\text{linear }  \mathrm{NUCA} \text{  with finite memory and uniformly bounded singularity}\}$. 
\end{enumerate}
\end{corollary} 

\begin{proof}
 
    Note also that every invertible NUCA with finite memory over a finite alphabet and a countable group universe is automatically stably invertible (see  Lemma~\ref{l:equivalent-stable-invertible}). We identify $A$ with its dual vector space $A^*$. Let $M\subset G$ be a finite subset. Let $S= \mathcal{L}(A^M, A)$ and let $s\in S^G$. By  identifying $A$ with its dual vector space $A^*$, note that $s^*\in T^G$,  where $T= \mathcal{L}(A^{M^{-1}}, A)$,  and $(s^*)^*=s$. According to  Proposition~\ref{p:prop-intro-dual}.(i)-(ii), $\sigma_s$ is stably injective$\iff$$\sigma_{s^*}$ is stably post-surjective; and $\sigma_s$ is stably invertible$\iff$so is $\sigma_{s^*}$. Consequently, the conclusion of the corollary follows for the class $C$ of all linear NUCA (a). 
By  Proposition~\ref{p:prop-intro-dual}.(iv), $\sigma_s$ has uniformly bounded singularity if and only if so does $\sigma_{s^*}$. Hence, the conclusion holds true for the class $C$ described in (c). Similarly, $s$  is asymptotically constant if and only if so is   $\sigma_{s^*}$ by Proposition~\ref{p:prop-intro-dual}.(v) and thus the  conclusion of the corollary follows also in the case (b).  
    \qed 
\end{proof}

Linear NUCA with finite memory enjoy similar properties as linear CA or more generally group CA such as the shadowing property  \cite{birkhoff-pseudo-orbit}, \cite{bowen-shadow-equilibrium}, \cite{kurka-97}, \cite{blanchard-maass-97},  \cite{meyerovitch-pseudo-orbit}, \cite{chung-shadow}, \cite{phung-israel}, \cite{phung-shadowing}, \cite{cscp-jpaa}, \cite{phung-laa}. 
The recent result \cite[Theorem D.(ii)]{phung-laa}  states that every stably post-surjective linear NUCA with finite vector space alphabet  over a residually finite group universe is invertible whenever it is a local perturbation of a linear CA. When specialized to the class of linear NUCA, Theorem~\ref{t:main-C} thus generalizes \cite[Theorem D.(ii)]{phung-laa} because every residually finite group is sofic and thus post-injunctive by Theorem~\ref{t:dual}. To this end, we postulate another seemingly interesting  question. 

\begin{question}
Does an affirmative  answer to Question~\ref{conjecture-stable-surj} imply an affirmative answer to  Question~\ref{Stable dual-surjunctivity}? Conversely, does an affirmative answer to Question~ \ref{Stable dual-surjunctivity} imply an affirmative answer to  Question~\ref{conjecture-stable-surj}?
\end{question}  

As a related problem, we can study analogues of our main results with  different  hypotheses on the configuration $s$ of local transition maps. For example, when the $G$-orbit of $s$ is finite, we obtain the following   result (see Section~\ref{s:finite-orbit}).

\begin{theoremletter}
\label{t:finite-orbit} 
 Let $M$ be a finite subset of a countable group~$G$. Let $A$ be a finite alphabet and let $S=A^{A^M}$. Suppose that $s \in S^G$ has  finite $G$-orbit in $A^G$, i.e., $|Gs|=|\{gs\colon g \in G\}|<\infty$. The following hold: 
 \begin{enumerate}[\rm (i)]
     \item  If $G$ is surjunctive and $\sigma_s $ is  injective, then $\sigma_s$ is invertible. 
     \item If $G$ is post-injunctive and  $\sigma_s$ is post-surjective, then $\sigma_s$ is invertible.
 \end{enumerate}  
\end{theoremletter}

It is also interesting to extend computable properties of linear  CA or more generally group CA \cite{Den-21}, \cite{Den-24}, \cite{Den-25} to the  setting of linear or group NUCA.

 \subsection{Organization of the paper}  

In Section~\ref{s:pre-1}, we prove  that the notions of invertibility and stable invertibility are equivalent for NUCA with finite memory (Lemma~\ref{l:equivalent-stable-invertible}). As a consequence, we  
relate the surjectivity, the invertibility,  and the stable invertibility of a stably injective NUCA in Corollary~\ref{r:remark-stably-invertible}. As another application, we show that pre-injective stably post-surjective NUCA with finite memory must be stably invertible. 
We then fix in Section~\ref{s:pre-2}  the notations and recall the construction of induced local maps of NUCA in Section~\ref{s:pre-2}. The proofs of Theorem~\ref{t:main-A}, Theorem~\ref{t:main-C}, Theorem~\ref{t:main-B},  Theorem~\ref{t:main-D}, and Theorem~\ref{t:finite-orbit} are given respectively in the subsequent Sections \ref{s:a}, \ref{s:b}, \ref{s:c},   \ref{s:d}, and \ref{s:finite-orbit}. The last expository Section~\ref{s:another-proof} contains a more self-contained and constructive proof of Theorem~\ref{t:main-C} in the case of finitely generated sofic group universes.

\section{Invertibility vs stable invertibility}
\label{s:pre-1}
\begin{lemma}
\label{l:equivalent-stable-invertible} 
Let $M$ be a finite subset of a countable group $G$. Let $A$ be a finite alphabet and let $s\in S^G$ where $S= A^{A^M}$. Suppose that $\sigma_s$ is invertible. Then $\sigma_s$ is stably invertible. 
\end{lemma}

\begin{proof}
Since $\sigma_s$ is invertible, there exist $N \subset G$ finite and $t \in T^G$ where $T=A^{A^N}$ such that $\sigma_s \circ \sigma_t = \sigma_t \circ \sigma_s=\Id_{A^G}$. For every $p \in \Sigma(s)$, we obtain from   the relation $\sigma_s \circ \sigma_t =\Id_{A^G}$ and \cite[Theorem 11.1]{phung-tcs}  some $q \in \Sigma(t)$ such that $\sigma_p \circ \sigma_q=\Id_{A^G}$. Since $q \in \Sigma(t)$ and $\sigma_t \circ \sigma_s=\Id_{A^G} $, another application of \cite[Theorem 11.1]{phung-tcs} shows that there exists $r \in \Sigma(s)$ such that $\sigma_q \circ \sigma_r=\Id_{A^G}$. It follows that $\sigma_q$ is invertible and thus $\sigma_p \circ \sigma_q=\sigma_q \circ \sigma_p=\Id_{A^G}$. Hence, $\sigma_s$ is stably invertible. \qed 
\end{proof}
\par 
\begin{corollary}
\label{r:remark-stably-invertible}
Let $G$ be a countable group and let $A$ be a finite alphabet. Let $\tau\colon A^G\to A^G$ be a stably injective NUCA with finite memory. Then the following are equivalent: 
\begin{enumerate}[\rm (i)]
    \item 
   $\tau$ is surjective, \item 
  $\tau$ is invertible, 
  \item 
  $\tau$ is stably invertible.
  \end{enumerate}
\end{corollary}
\begin{proof}
    It is trivial that (iii)$\implies$(i). The implication (i)$\implies$(ii) results from the  definition of invertibility and  \cite[Theorem~A]{phung-tcs}. Lemma~\ref{l:equivalent-stable-invertible} states  that (ii)$\implies$(iii) and the proof is thus complete. \qed  
\end{proof}

\begin{corollary}
\label{c:post-surj-invertible}
Let $G$ be a countable group and let $A$ be a finite alphabet. Let $\tau\colon A^G\to A^G$ be a  
  pre-injective stably post-surjective NUCA with finite memory. Then $\tau$  is  stably invertible. 
\end{corollary}

\begin{proof}
The fact that $\tau$ is invertible follows from  \cite[Theorem 13.4]{phung-tcs} (see also \cite{kari-post-surjective}). We can thus conclude from  Lemma~\ref{l:equivalent-stable-invertible} that $\tau$ is stably invertible. \qed   
\end{proof}

\section{Induced local  maps of NUCA} 
\label{s:pre-2}
Let $G$ be a group and let $A$ be an alphabet. For every subset $E\subset G$ and $x \in A^E$ we define  $gx \in A^{gE}$ by  $gx(gh)=x(h)$ for all $h \in E$. In particular,    $gA^E=A^{gE}$. 
Let $M\subset G$ and  let $S=A^{A^M}$ be the collection of all maps $A^M \to A$. 
For every finite subset $E \subset G$  and $w \in S^{E}$,  
we define a map  $f_{E,w}^{+M} \colon A^{E M} \to A^{E}$  as follows. For every $x \in A^{EM}$ and $g \in E$, we set: 
\begin{align}
\label{e:induced-local-maps} 
    f_{E,w}^{+M}(x)(g) & = w(g)((g^{-1}x)\vert_M). 
\end{align}
\par 
\noindent 
In the above formula, note that  $g^{-1}x \in A^{g^{-1}EM}$ and $M \subset g^{-1}EM$ since $1_G \in g^{-1}E$ for $g \in E$. Therefore, the map   $f_{E,w}^{+M} \colon A^{E M} \to A^{E}$ is well defined. 
Consequently, for every $s \in S^G$, we have a well-defined induced local map $f_{E, s\vert_E}^{+M} \colon A^{E M} \to A^{E}$ for every finite subset $E \subset G$ which satisfies: 
\begin{equation}
\label{e:induced-local-maps-general} 
    \sigma_s(x)(g) =  f_{E, s\vert_E}^{+M}(x\vert_{EM})(g)
\end{equation}
for every $x \in A^G$ and $g \in E$. Equivalently, we have for all $x \in A^G$ that: 
\begin{equation}
\label{e:induced-local-maps-proof} 
    \sigma_s(x)\vert_E =  f_{E, s\vert_E}^{+M}(x\vert_{EM}). 
\end{equation}

For every $g \in G$, we have a canonical  bijection  $\gamma_g\colon G \mapsto G$ induced by the translation $a \mapsto g^{-1}a$. For each subset $K \subset G$, we denote by $\gamma_{g, K} \colon gK \to K$ the restriction to $gK$ of $\gamma_g$. 
Now let $N\subset G$ and $T=A^{A^N}$. Let $t \in T^G$. With the above notations, we have the following auxiliary lemma.  

\begin{lemma}
    \label{l:compo-id} Suppose that $1_G\in M\cap N$. Then for every $g \in G$, the condition $\sigma_t(\sigma_s(x))(g)=x(g)$ for all $x \in A^G$ is equivalent to the condition 
    $$t(g)\circ  \gamma_{g, N} \circ f^{+M}_{gN, s\vert_{gN}}\circ \gamma^{-1}_{g, NM}=\pi,$$
    where $\pi\colon A^{NM} \to A$ is the projection $z\mapsto z(1_G)$. 
\end{lemma}

\begin{proof}
  For every $g \in G$ and $x \in A^G$, we deduce from Definition~\ref{d:most-general-def-asyn-ca} and the relation \eqref{e:induced-local-maps-proof} that 
  \begin{align*}
     \sigma_t(\sigma_s(x))(g) & = t(g) \left( (g^{-1}\sigma_s(x))\vert_N\right)\\
     & = t(g) \left( \gamma_{g,N}\left( (\sigma_s(x))\vert_{gN}\right)\right)\\
     & =  t(g) \left( \gamma_{g,N}\circ f^{+M}_{gN, s\vert_{gN}}(x\vert_{gNM})\right) \\
     & =  t(g) \circ \gamma_{g,N}\circ f^{+M}_{gN, s\vert_{gN}}(x\vert_{gNM}) \\
     & =  t(g) \circ  \gamma_{g,N}\circ f^{+M}_{gN, s\vert_{gN}}\circ \gamma^{-1}_{g, NM}\left((g^{-1}x)\vert_{NM}\right) 
  \end{align*}
  from which the conclusion follows as $x(g)= (g^{-1}x)(1_G)$. 
\qed    \end{proof}

\section{Proof of Theorem~\ref{t:main-A}} 
\label{s:a}

We give below the proof of the stable invertibility of  stably injective NUCA which are local perturbations of CA over a surjunctive group universe. 

 \begin{proof}
 
Theorem~\ref{t:main-A} is trivial if $G$ is finite since every injective selfmap of a finite set is also surjective. Thus, without loss of generality we can suppose that $G$ is infinite.  Up to enlarging $M$ if necessary, we can also assume that $1_G\in M$. 
As $s\in S^G$ is asymptotically constant, we can find a constant configuration $c\in S^G$ and a finite subset $F\subset G$ such that $s\vert_{G \setminus F} =  c\vert_{G \setminus F}$. Note that as $G$ is infinite, we have $c\in \Sigma(s) $. 
Since $\sigma_s$ is stably injective, we deduce that  $\sigma_c\colon A^G\to A^G$ is an injective CA. We infer from the surjunctivity of the group $G$ that $\sigma_c$ is also surjective. By Corollary~\ref{r:remark-stably-invertible} and Remark~\ref{r:remark}, it follows that $\sigma_c$ is invertible. Therefore,  
there exist a nonempty finite subset $N\subset G$ 
and a constant configuration $d \in T^G$, where $T=A^{A^N}$, such that $1_G\in N$ and  
\begin{equation}
    \label{e:theorem-a-1}
\sigma_c\circ \sigma_d=\sigma_d\circ \sigma_c=\Id_{A^G}. 
\end{equation}
 
\par 
The set of NUCA with finite memory over the universe $G$ and the alphabet $A$ forms a monoid  with respect to the composition operation and we obtain from \cite[Theorem~6.2]{phung-tcs}  a configuration $q\in Q^G$, where $Q=A^{A^{MN}}$, such that $\sigma_s \circ \sigma_d=\sigma_q$. Note that $MN$ is a memory set of $\sigma_q$. 
As both $\sigma_s$   and $\sigma_d$ are  injective,  $\sigma_q=\sigma_s \circ \sigma_d$ is  also  injective. 
Since  $\sigma_c \circ \sigma_d=\Id_{A^G}$ and  $\sigma_s \circ \sigma_d=\sigma_q$ and  $s$ is asymptotic to $c$, we deduce that the configuration $q \in Q^G$ is asymptotic to the constant configuration $\pi^G$ where $\pi\colon A^{MN} \to A^{\{1_G\}}$ is the canonical projection $x\mapsto x(1_G)$. More precisely, $q\vert_{G\setminus F}=\pi^{G\setminus F}$ since  $s\vert_{G\setminus F}=c\vert_{G\setminus F}$ so that for every $g \in G\setminus F$ (see \eqref{s:pre-2} and the proof of  \cite[Theorem~6.2]{phung-tcs}): 
\begin{equation}
\label{e:proof-A-id-F}
q(g) = s(g) \circ f^{+N}_{M, g^{-1}d\vert_{M}} = c(g) \circ f^{+N}_{M, g^{-1}d\vert_{M}} = \pi 
\end{equation}
where the last equality results from \eqref{e:theorem-a-1}. 
\par 
Let $E=FMN\subset G$ then $E$ is  finite and $F\subset E$ as $1_G\in M\cap N$. It follows that    $q\vert_{G\setminus E} = \pi^{G\setminus E}$. 
Consider the map $\Phi\colon A^{E} \to A^E$ induced by the restriction of $\sigma_p$ to $A^E$. More specifically, for every $x \in A^E$, we define   
\begin{equation}
\label{e:proof-A-id-F-2}
\Phi(x)=\sigma_q(y)\vert_{E}
\end{equation}
for any configuration $y \in A^G$ extending $x$, that is, $y\vert_E=x$. To check that $\Phi$ is well-defined, let $z\in A^G$ be another configuration such that $z\vert_E=x$. Let $g \in E\setminus F$ then we have $q(g)=\pi$ and thus 
\begin{align*}
    \sigma_q(z)(g)& =  q(g)((g^{-1}z)  
	\vert_{MN}) = \pi ((g^{-1}z)  
	\vert_{MN}) \\
    &= z(g)=x(g) = y(g) \\& = \pi ((g^{-1}y)  
	\vert_{MN})  =q(g)((g^{-1}y)  
	\vert_{MN})  \\& 
    = \sigma_q(y)(g). 
\end{align*}
Now let $g \in F$. Then $(g^{-1}z)\vert_{MN}= (g^{-1}y)  
	\vert_{MN}$ since $(g^{-1}z)(h) = z(gh)=x(gh)=y(gh)= (g^{-1}y)(h)$ for all $h\in MN$. Therefore,  
 \begin{align*}
    \sigma_q(z)(g)& =  q(g)((g^{-1}z)  
	\vert_{MN}) =  q(g)((g^{-1}z)  
	\vert_{MN}) = q(g)((g^{-1}u)  
	\vert_{MN}) =  \sigma_q(z)(g). 
    \end{align*}
We conclude that $  \sigma_q(z)\vert_E=  \sigma_q(y)\vert_E$ and thus $\Phi$ is well-defined. 
\par 
As $MN$ is a memory set of $\sigma_q$,  we infer from   \eqref{e:proof-A-id-F} and the relation $E=FMN$   that  $\sigma_q(x)\vert_{G\setminus E}=x\vert_{G\setminus E}$ for every $x \in A^G$.  
Combining with   \eqref{e:proof-A-id-F-2}, we obtain 
$$
\sigma_q = \Phi \times \Id_{A^{G \setminus E}}\colon A^E \times A^{G\setminus E} \to A^E \times A^{G\setminus E}  
$$
because  for every   $x=(y,z) \in A^E \times A^{G\setminus E}=A^G$ where $y=x\vert_E$ and $z=x\vert_{G\setminus E}$, we have  $\sigma_q(x)=(\sigma_q(x)\vert_E, \sigma_q(x)\vert_{G\setminus E})=(\Phi(y), z) \in A^G$. We note here that such a map $\sigma_q$ is related to the notion of a \emph{gate} on a subshift \cite{salo}[Lemma 15].  

\par 
Since $\sigma_q$ and $\Id_{A^{G \setminus E}}$ are injective, we deduce that $\Phi$ is also injective. Consequently, $\Phi$ must be surjective as $A^E$ is finite. It follows that $\sigma_q= \Phi \times \Id_{A^{G \setminus E}}$ is  surjective. Combining with \eqref{e:theorem-a-1} and the surjectivity of $\sigma_c$, we find  that $\sigma_s= \sigma_q\circ (\sigma_d)^{-1}=\sigma_q\circ \sigma_c$ is also surjective. Since  $\sigma_s$ is stably injective by hypothesis, we can thus conclude from  Corollary~\ref{r:remark-stably-invertible} that  $\sigma_s$ is stably invertible and the proof is complete. \qed 
 \end{proof}

\section{Proof of Theorem~\ref{t:main-C}} 
\label{s:b}

By following the same lines, \emph{mutatis mutandis}, as in the proof of Theorem~\ref{t:main-A}, we obtain below the proof of Theorem~\ref{t:main-C} on the stable invertibility of stably post-surjective NUCA which are local perturbations of CA over post-injunctive group universes.

 \begin{proof}
 We can suppose without loss of generality that $G$ is infinite since  every surjective selfmap of a finite set is also injective.  Up to enlarging $M$, we can also assume that $1_G\in M$. 
By hypothesis, there exist a constant configuration $c\in S^G$ and a finite subset $F\subset G$ with  $s\vert_{G \setminus F} =  c\vert_{G \setminus F}$. Note again that $c\in \Sigma(s)$ as $G$ is infinite. 
We infer from the stable post-surjectivity of $\sigma_s$ that the CA  $\sigma_c\colon A^G\to A^G$ is post-surjective. Since $G$ is post-injunctive,  $\sigma_c$ is pre-injective  and thus invertible by Corollary~\ref{c:post-surj-invertible} and Remark~\ref{r:remark}. Hence,  
we can find finite subset $N\subset G$ 
and a constant configuration $d \in T^G$, where $T=A^{A^N}$, such that $1_G\in N$ and  
\begin{equation}
    \label{e:theorem-a-2}
\sigma_c\circ \sigma_d=\sigma_d\circ \sigma_c=\Id_{A^G}. 
\end{equation}
 
\par 
By \cite[Theorem~6.2]{phung-tcs}, there exists  a configuration $q\in Q^G$, where $Q=A^{A^{MN}}$, such that $\sigma_s \circ \sigma_d=\sigma_q$.  
As $\sigma_s$   and $\sigma_d$ are  surjective, so is $\sigma_q=\sigma_s \circ \sigma_d$. 
Let $\pi\colon A^{MN} \to A^{\{1_G\}}$ be the canonical projection $x\mapsto x(1_G)$. As $s\vert_{G\setminus F}=c\vert_{G\setminus F}$, we have  $q\vert_{G\setminus F}=\pi^{G\setminus F}$ since  for every $g \in G\setminus F$  (see the proof of  \cite[Theorem~6.2]{phung-tcs}): 
$$
q(g) = s(g) \circ f^{+N}_{M, g^{-1}d\vert_{M}} = c(g) \circ f^{+N}_{M, g^{-1}d\vert_{M}} = \pi
$$
where the last equality follows from \eqref{e:theorem-a-2}. 
Let $E=FMN\subset G$ then $F\subset E$ as $1_G\in M\cap N$. It follows that    $q\vert_{G\setminus E} = \pi^{G\setminus E}$. 
As in the proof of Theorem~\ref{t:main-A}, we can write $ 
\sigma_q = \Phi \times \Id_{A^{G \setminus E}} 
$ where  $\Phi\colon A^{E} \to A^E$ is the map given by 
$ 
\Phi(x)=\sigma_q(y)\vert_{E} 
$ 
for every $x \in A^E$ and any  $y \in A^G$ such that $y\vert_E=x$. 
 Then  $\Phi$ is surjective because   $\sigma_q$ and $\Id_{A^{G \setminus E}}$ are both surjective. It follows that $\Phi$ must be injective as $A^E$ is finite. 
 Consequently, $\sigma_q= \Phi \times \Id_{A^{G \setminus E}}$ is  injective. Thus by \eqref{e:theorem-a-2}, we have $\sigma_s= \sigma_q\circ (\sigma_d)^{-1}=\sigma_q\circ \sigma_c$ is injective and thus pre-injective. By the  stable post-surjectivity of  $\sigma_s$ and Corollary~\ref{c:post-surj-invertible}, we conclude that  $\sigma_s$ is stably invertible and the proof is thus complete. \qed 
 \end{proof}

\section{Proof of Theorem~\ref{t:main-B}} 
\label{s:c} 
The goal of this section is to establish Theorem~\ref{t:main-B} which states that over surjunctive group universes,  every stably
injective NUCA with finite memory and uniformly bounded singularity must be stably
invertible.  
We first prove the following useful technical lemma which  will enable a reduction of  Theorem~\ref{t:main-B} to Theorem~\ref{t:main-A}.  

\begin{lemma}
\label{l:main-lemma-singular}
Let $A$ be an alphabet and let $G$ be a finitely generated group. Let $M\subset G$ be a finite subset and $S=A^{A^M}$. Suppose that $\sigma_t \circ \sigma_s=\Id_{A^G}$ for some $s,t \in S^G$ and $s$ has uniformly bounded singularity. Then 
for each $E \subset G$ finite, there exist asymptotically constant configurations $p,q\in S^G$ such that $p\vert_E = s\vert_E$, $q \vert_E= t\vert_E$, and $\sigma_q \circ \sigma_p=\Id_{A^G}$. 
\end{lemma}

\begin{proof}
Up to enlarging $M$, we can assume that $1_G \in  M=M^{-1}$.  
Let $E\subset G$ be a finite subset. Up to enlarging $E$, we can suppose without loss of generality that $ M  \subset E =E^{-1}$ and $E$ is a finite generating set of $G$.  As $s$ has uniformly bounded singularity, there exist a constant configuration $c \in S^G$  and a finite subset $F\subset G$ containing $E^3$ such that $s\vert_{FE^3\setminus F}= c\vert_{FE^3\setminus F}$. We define an asymptotically constant configuration $p\in S^G$ by setting $p\vert_{FE}=s\vert_{FE}$ and $p\vert_{G\setminus FE}=c\vert_{G\setminus FE}$. 
Note that $ FE^2 \setminus FE \neq \varnothing$. Indeed,    if $FE^2\subset FE$ then by induction $FE^n \subset FE$ for every integer $n \geq 1$ and thus $G= \bigcup_{n\geq 1}E^n \subset \bigcup_{n\geq 1}FE^n\subset FE$ which is a contradiction since $G$ is infinite. 
Therefore, we can fix $g_0 \in FE^2 \setminus FE$ and 
  define $q \in S^G$ by  $q(g)=t(g_0)$ if $g \in G\setminus FE$ and $q(g)=t(g)$ if $g \in FE$. Then $q$ is asymptotic to the constant configuration $d \in S^G$ defined by $d(g)= t(g_0)$ for all $g \in G$. Since $1_G \in M$, we have a projection $\pi \colon A^{M^2}\to A$ given by $x\mapsto x(1_G)$. Since $E\subset E^3 \subset F \subset FE$, we deduce from our construction that $p\vert_E =s\vert_E$ and $q\vert_E =t\vert_E$. To conclude, we only need to check that $\sigma_q\circ \sigma_p=\Id_{A^G}$.
  \par 
  Let $g \in FE^2 \setminus FE$. Then $gE \subset FE^3\setminus F$ since $E=E^{-1}$.  Consequently,  
$s\vert_{gE}= c\vert_{gE}=p\vert_{gE}$ and thus $s\vert_{gM}= c\vert_{gM}=p\vert_{gM}$   since $M \subset E$. Hence, the condition $\sigma_t( \sigma_s(x))(g)=x(g)$ for all $x \in A^G$ is equivalent to $t(g) \circ f^{+M}_{M, c\vert_M} = \pi$ by Lemma~\ref{l:compo-id}. Similarly,  the condition $\sigma_q( \sigma_p(x))(g)=x(g)$ for all $x \in A^G$  amounts to $q(g) \circ f^{+M}_{M, c\vert_M} = \pi $. Since $q(g)=t(g_0)$ and $\sigma_t\circ \sigma_s=\Id_{A^G}$, we conclude from the above discussion that $\sigma_q( \sigma_p(x))(g)=x(g)$ for all $x \in A^G$. 
 \par 
 Let  $g \in  FE$. 
 Since $s\vert_{FE^3\setminus F}= c\vert_{FE^3\setminus F}$,   $p\vert_{FE}=s\vert_{FE}$, and $p\vert_{G\setminus FE}=c\vert_{G\setminus FE}$ by construction, we have  $p\vert_{FE^3}=s\vert_{FE^3}$. In particular, $p\vert_{gM}=s\vert_{gM}$ since $gM \subset (FE)E=FE^2 \subset FE^3$.  Therefore, we can infer from the relations  $\sigma_t\circ \sigma_s=\Id_{A^G}$ and  $q(g)=t(g)$  that $\sigma_q( \sigma_p(x))(g)=x(g)$ for all $x \in A^G$.
 \par 
 Finally, let $g \in G \setminus FE^2$. Since $p\vert_{G\setminus FE}=c\vert_{G\setminus FE}$,  $q\vert_{G\setminus FE}=d\vert_{G\setminus FE}$, and since $c,d$ are constant, we deduce that $q(g)= d(g_0)=t(g_0)$ and $p\vert_{gM}= c\vert_{gM}$. 
 The condition $\sigma_q( \sigma_p(x))(g)=x(g)$ for all $x \in A^G$ is thus  equivalent to $t(g_0) \circ f^{+M}_{M, c\vert_M} = \pi$ by Lemma~\ref{l:compo-id}. But since $\sigma_t\circ\sigma_s=\Id_{A^G}$ and $s\vert_{g_0M}=c\vert_{g_0M}$, another application of  Lemma~\ref{l:compo-id} shows that $t(g_0) \circ f^{+M}_{M, c\vert_M} = \pi$. 
 Thus, $\sigma_t( \sigma_s(x))(g)=x(g)$ for all $x \in A^G$. 
Therefore, we conclude that $\sigma_q\circ \sigma_p=\Id_{A^G}$. The proof is complete.  \qed 
\end{proof}

We are now in position to prove Theorem~\ref{t:main-B}. 
\par 
\vspace{0.3cm}
\noindent 
\textit{Proof of Theorem~\ref{t:main-B}}. 
    Since $\sigma_s$ is stably injective by hypothesis, we deduce from \cite[Theorem~A]{phung-tcs} that there exist a finite subset $N \subset G$ and $t \in T^G$, where $T=A^{A^N}$, such that $\sigma_t\circ \sigma_s= \Id_{A^G}$. Up to enlarging $M$ and $N$, we can assume that $1_G\in M=N$ and thus $S=T$. By Corollary~\ref{r:remark-stably-invertible}, it suffices to show that $\sigma_s$ is surjective to conlude that $\sigma_s$ is stably invertible.  We suppose on the contrary that $\sigma_s$ is not surjective. Since $\Gamma=\sigma_s(A^G)$ is closed in $A^G$ with respect to the prodiscrete topology by \cite[Theorem~4.4]{phung-tcs}, there must exist a nonempty finite subset $E \subset G$ such that $\Gamma_{E} \subsetneq A^E$. 
    Since $s$ has uniformly bounded singularity, we infer from Lemma~\ref{l:main-lemma-singular} that there exist asymptotically constant configurations $p,q\in S^G$ such that $p\vert_{E} = s\vert_{E
    }$, $q \vert_{E}= t\vert_{E}$, and $\sigma_q \circ \sigma_p=\Id_{A^G}$. Let $\Lambda=\sigma_p(A^G)$. Then it follows from $p\vert_{E} = s\vert_{E}$ that $\Lambda_E=\Gamma_E\subsetneq A^E$. We deduce that $\sigma_p$ is not surjective. In particular, $\sigma_p$ is not invertible. On the other hand, the condition $\sigma_q \circ \sigma_p=\Id_{A^G}$ shows that $\sigma_p$ is stably injective by \cite[Theorem~A]{phung-tcs}. We can thus apply Theorem~\ref{t:main-A} to deduce that $\sigma_p$ is invertible and thus surjective. Therefore, we obtain a contradiction and the proof is complete. \qed

\section{Proof of Theorem~\ref{t:main-D}}
\label{s:d}
In this section, we prove that 
stably post-surjective NUCA with finite memory and  uniformly bounded singularity is stably invertible whenever the group universe is post-injunctive. 
The following proof of Theorem~\ref{t:main-D}   generalizes the proof of  Theorem~\ref{t:main-A} and Theorem~\ref{t:main-C}.    
\begin{proof}
We will show that $\sigma_s$ is injective. Suppose on the contrary that there exist an element $h \in G$ and configurations $u,v \in A^G$ such that  $\sigma_s(u)= \sigma_s(v)$ but $u(h)\neq v(h)$.     
Up to enlarging $M$,   
we can suppose without loss of generality that $M=M^{-1}$ (that is,  $M$ is symmetric) and $1_G, h \in M$. When $G$ is finite, $\sigma_s$ is trivially invertible as a surjective selfmap of a finite set. Hence, we assume in what follows that the group $G$ is infinite. 
\par 
Let $c_1, ..., c_n \in S^G$ be all the constant configurations in $S^G$ where $n=|S|$. Let $\Delta \subset G$ be a finite  generating set of $G$ such that $1_G \in \Delta= \Delta^{-1}$ (that is, $\Delta$ is symmetric) and $M \subset \Delta$. Up to enlarging $M$, we can suppose that $M=\Delta$. For each $k \geq 1$, let $E_k=\Delta^k$ then $E_k$ is  a finite symmetric subset of $G$ containing $1_G$ and we have an exhaustion $G= \bigcup_{k =1}^\infty E_k$. 
Since $s$ has uniformly bounded singularity, we can find a finite subset  $F_k \subset G$  such that $E_k^5\subset F_k$ and  $s\vert_{F_k E^5_k \setminus F_k} \in S^{F_k E^5_k \setminus F_k}$ is constant for every $k \geq 1$. Since $S$ is finite, there exist $m \in \{1,2,..., n\}$ and an infinite sequence $1\leq k_1<k_2<k_3<\dots$ of integers such that for $c=c_m$, we have  $ E^5_{k_i}\subset F_{k_i}$ for every $i \geq 1$ and 
\begin{align}
\label{e:proof-main-B1-1}
s\vert_{F_{k_i} E^5_{k_i} \setminus F_{k_i}} = c\vert_{F_{k_i} E^5_{k_i} \setminus F_{k_i}}.
\end{align} 
\par 
For every $i \geq 1$, we claim that $F_{k_i} E^2_{k_i} \setminus F_{k_i}E_{k_i}\neq \varnothing$. Indeed, we would have otherwise $F_{k_i} E^2_{k_i} \subset  F_{k_i} E_{k_i}$. Hence, $F_{k_i} E^r_{k_i} \subset F_{k_i} E_{k_i}$  by induction on $r\geq 2$ and 
$$G=F_{k_i}  G= F_{k_i}  \bigcup_{r \geq 2} E^r_{k_i}  = \bigcup_{r \geq 2} F_{k_i}  E^r_{k_i}  \subset F_{k_i} E_{k_i}$$
which is a contradiction since $G$ is infinite. Therefore, for every $i\geq 1$, we can fix some $g_i \in F_{k_i} E^2_{k_i} \setminus F_{k_i}E_{k_i}$. As $1_G \in E_{k_i}$ and $E_{k_i}$ is symmetric, it follows that 
\begin{align}
\label{e:proof-main-B1-2}
g_i E_{k_i}\subset   F_{k_i} E^3_{k_i} \setminus F_{k_i}. 
\end{align}  
 
\par 
It follows from \eqref{e:proof-main-B1-1} and \eqref{e:proof-main-B1-2} that $c\vert_{g_iE_{k_i}} = s\vert_{g_iE_{k_i}}$. Since  $G=\bigcup_{i=1}^\infty E_{k_i}$ and $\Sigma(s) = \overline{\{gs \colon g \in G\}} \subset S^G$, we deduce that $c\in \Sigma(s)$. We thus infer from the stable post-surjectivity of $\sigma_s$ that  $\sigma_c\colon A^G \to A^G$ is a post-surjective CA. Hence, $\sigma_c$ is pre-injective by the post-injunctivity of the group universe $G$. Theorem~\ref{t:main-B} then implies that $\sigma_s$ is an invertible CA. Therefore, there exists a CA $\tau\colon A^G \to A^G$ such that 
\begin{equation}
    \label{e:proof-d-2}
\tau\circ \sigma_c=\sigma_c\circ \tau=\Id_{A^G}.
\end{equation} 
Without loss of generality, we can assume that $M$ is also a memory set of $\tau$ up to enlarging $M$. Thus $\tau=\sigma_d$ for some constant configuration $d\in S^G$.  
\par 
 
From \cite[Theorem~6.2]{phung-tcs} we obtain  a configuration $q\in Q^G$, where $Q=A^{A^{M^2}}$, such that $\sigma_d \circ \sigma_s=\sigma_q$.  
Let $\pi \colon A^{M^2} \to A^{\{1_G\}}$ denote the canonical projection $z \mapsto z(1_G) $ and fix  $j\geq 1$ large enough such that $M^2 \subset E_{k_j}$.  
We claim that 
\begin{align}
\label{e:proof-d-10}
q\vert_{F_{k_j}E^4_{k_j}\setminus F_{k_j}E_{k_j}}=\pi^{F_{k_j}E^4_{k_j}\setminus F_{k_j}E_{k_j}}.
\end{align}
Indeed,  let  $g \in F_{k_j}E^4_{k_j}\setminus F_{k_j}E_{k_j}$. Then  $ gE_{k_j}\subset F_{k_j}E^5_{k_j}\setminus F_{k_j}$ as $E_{k_j}$ is symmetric. Hence, $gM\subset gE_{k_j}\subset F_{k_j}E^5_{k_j}\setminus F_{k_j}$. It follows that  $s\vert_{gM} = c\vert_{gM}$ by \eqref{e:proof-main-B1-1} and thus $
g^{-1}s \vert_{M}=g^{-1}c \vert_{M}$. Since  $\sigma_d\circ \sigma_c=\Id_{A^G}$ by \eqref{e:proof-d-2}, the claim \eqref{e:proof-d-10} follows since  
(see  the proof of  \cite[Theorem~6.2]{phung-tcs}): 
$$
q(g) = d(g) \circ f^{+M}_{M, g^{-1}s\vert_{M}} = d(g) \circ f^{+M}_{M, g^{-1}c\vert_{M}} = \pi.
$$ 
Therefore, as in the proof of Theorem~\ref{t:main-A}, we deduce from \eqref{e:proof-d-10}  that  
\begin{equation}
\label{e:decomposition-last}
\sigma_q = \Phi \times \left ( \Id_{A^{F_{k_j}E^3_{k_j}\setminus F_{k_j}E^2_{k_j}}} \right)  \times \Psi 
\end{equation}
where  $\Phi\colon A^{F_{k_j}E^2_{k_j}} \to A^{F_{k_j}E^2_{k_j}}$ and $\Psi\colon A^{G\setminus F_{k_j}E^3_{k_j}} \to A^{G\setminus F_{k_j}E^3_{k_j}}$ are well-defined maps given by the following formula:  
\begin{enumerate}[\rm (i)]
    \item 
$ 
\Phi(x)=\sigma_q(y)\vert_{F_{k_j}E^2_{k_j}} 
$ 
for  $x \in A^{F_{k_j}E^2_{k_j}}$ and   $y \in A^G$ with $y\vert_{F_{k_j}E^2_{k_j}}=x$, 
\item 
$ 
\Psi(x)=\sigma_q(y)\vert_{G\setminus F_{k_j}E^3_{k_j}} 
$ 
for  $x \in A^{G\setminus F_{k_j}E^3_{k_j}}$ and  $y \in A^G$ with  $y\vert_{G\setminus F_{k_j}E^3_{k_j}}=x$.
\end{enumerate}
\par 
 Recall the choice of the configurations $u,v\in A^G$ in the first paragraph of the proof with $u(h)\neq v(h)$ and $\sigma_s(u)=\sigma_s(v)$ for some $h \in M \subset F_{k_j}E^2_{k_j}$. It follows that $u\vert_{F_{k_j}E^2_{k_j}} \neq v\vert_{F_{k_j}E^2_{k_j}}$ and 
 \begin{align*}
\Phi\left(u\vert_{F_{k_j}E^2_{k_j}}\right) = \sigma_q(u)\vert_{F_{k_j}E^2_{k_j}} & = \sigma_d\left(\sigma_s(u)\right)\vert_{F_{k_j}E^2_{k_j}} \\& =  \sigma_d\left(\sigma_s(v)\right)\vert_{F_{k_j}E^2_{k_j}}= \sigma_q(v)\vert_{F_{k_j}E^2_{k_j}}=\Phi\left(v\vert_{F_{k_j}E^2_{k_j}}\right). 
 \end{align*}
We deduce that $\Phi$ is not injective and thus it is not surjective as a selfmap of the finite set $A^{F_{k_j}E^2_{k_j}}$. Consequently, \eqref{e:decomposition-last} implies that $\sigma_q$ is not surjective. We arrive at a contradiction since $\sigma_q=\sigma_d\circ \sigma_s$ is surjective as a composition of surjective maps. Hence,  $\sigma_s$ must be injective. By Corollary \ref{c:post-surj-invertible}, we conclude that $\sigma_s$ is stably invertible and the proof is complete. \qed
\end{proof}

 \section{Proof of Theorem~\ref{t:finite-orbit}}

 \label{s:finite-orbit}
 
 To recall the notations, let $M \subset G$ be a finite subset. Let $A$ be a finite alphabet and let $S=A^{A^M}$. Let $s\in S^G$ be a configuration of local transition maps with finite orbit. 
Since the orbit $Gs=\{gs\colon g \in G\}$ is finite, we have  $\Sigma(s)=Gs$. For every $g\in G$, we know that $\sigma_{gs}$ is injective, resp. post-surjective, resp. invertible, if and only if so is $\sigma_s$ (see \cite[Lemma~5.1]{phung-tcs}). Therefore, $\sigma_s$ is stably injective, resp. stably post-surjective, resp. stably invertible, if and only if it is  injective, resp. post-surjective, resp. invertible.
\par 
Let $H=\mathrm{Stab}_G(s)\subset G$ be the stabilizer subgroup of $s$ in $G$. Note that $|H\backslash G|= |Gs|$ so that $H$ is a subgroup of finite index in $G$.     
Let $B=A^{H\setminus G}$. We will construct a CA $\tau\colon B^H \to B^H$ induced by $\sigma_s$ as follows. 
Fix a complete set of representatives $g_1,\dots, g_n \in G$ of right cosets of $H$ in $G$, that is, $H\backslash G= \{Hg_1, \dots, Hg_n\}$. 
For every $x \in B^H$, we define  $\phi_H(x)\in A^G$ by setting  \begin{align} 
\label{e:phi-H}
\phi_H(x)(g)=x(h)(Hg_i)\in A
\end{align}
 for every $g \in G$ where we write $g=hg_i$ for some unique $h \in H$ and $i \in \{1,\dots, n\}$. The map $\phi_H \colon B^H \to A^G$ is  $H$-equivariant  since for every $x\in B^H$,  $h, k \in H$, and $g = hg_i \in G$  with $i \in\{1,\dots, n\}$, we have by definition of $\phi_H$ that  
\begin{align*}
    \phi_H(kx)(g)  & = (kx)(h)(Hg_i) = x(k^{-1}h)(Hg_i)  \\ & =\phi_H(x)(k^{-1}h g_i) = ( k \phi_H(x))(h g_i) \\
   & = (k\phi_H(x))(g)
\end{align*} 
and thus $\phi_H(k x)=k\phi_H(x)$ for all $x\in B^H$ and $k \in H$. The map $\phi_H$ is also a bijection. Indeed, for every $y \in A^G$, let $\psi_H \colon A^G\to B^H$ be the map defined by 
\begin{align} 
\label{e:psi-H}
\psi_H(y)(h)(Hg_i) = y(hg_i)
\end{align}
for every $y \in A^G$, $h\in H$, and   $i \in \{1,\dots, n\}$. It can be checked directly  that  $$\psi_H\circ \phi_H= \Id_{A^G} \quad \text{ and }\quad  \psi_H\circ \phi_H= \Id_{B^H}.$$  

\par 
Consider the map  $\tau \colon B^H \to B^H$ defined for all $x \in B^H$  by  $\tau = \psi_H\circ \sigma_s \circ \phi_H$. 
Let  $x\in B^H$ and $y = \phi_H(x)$. Then  $\tau(x)= \psi_H(\sigma_s(y))$. Fix  $h \in H$  and $i \in \{1,\dots, n\}$. Then we have $s= h^{-1}s$ since $h^{-1}\in H = \mathrm{Stab}_{G} (s)$. In particular, $s(g_i)=(h^{-1}s)(g_i) = s(hg_i)$. We find that 
\begin{align*}
   \tau(x)(h)(Hg_i)  &  = \psi_H(\sigma_s(y))(h) (Hg_i) & \\
   & =  \sigma_s(y) (hg_i) & \text{(by } \eqref{e:psi-H}) \\& 
   = s(hg_i) (((hg_i)^{-1}y) \vert_{M}) & \text{(by Definition }\ref{d:most-general-def-asyn-ca})\\
   & = s(g_i) (((hg_i)^{-1}y) \vert_{M}) \quad & \text{(as } s(hg_i)=s(g_i))\\
   & =  s(g_i) ((g_i^{-1}h^{-1}\phi_H(x))) \vert_{M}) & \text{(as } y= \phi_H(x))\\
   & = s(g_i) ((g_i^{-1} \phi_H(h^{-1}x)) \vert_{M})  &  \text{(as }\phi_H \text{ is }H\text{-equivariant})
\end{align*}
For every $g \in M$,  we can write $g_ig=h_{g_ig} g_{j_{g_ig}}$ for some unique $j_{g_ig} \in \{1,\dots, n\}$ and $h_{g_ig}\in H$. Consider the finite subset $N= \{h_{g_ig}: 1\leq i\leq n, \, g \in M\} \subset H$. By \eqref{e:phi-H}, we have  for  all $g \in M$ that  $$(g_i^{-1} \phi_H(h^{-1}x)(g)= \phi_H(h^{-1}x)(g_ig)=(h^{-1}x)(h_{g_ig})(Hg_{j_{g_ig}}). 
$$
Hence, $(g_i^{-1} \phi_H(h^{-1}x))\vert_M$ depends only on $(h^{-1}x)\vert_N$. Since  $$ \tau(x)(h)(Hg_i)  = s(g_i) ((g_i^{-1} \phi_H(h^{-1}x)) \vert_{M})$$ for all $x\in B^H$, $h \in H$ and $i \in \{1,\dots, n\}$  and $H\backslash G= \{Hg_1,\dots, Hg_n\}$, it follows that $\tau(x)(h)\in B= A^{H\backslash G}$  depends only on $(h^{-1}x)\vert_{N}$.  Consequently, $\tau \colon B^H \to B^H$ is a CA with finite memory  $N\subset H$. To summarize, we have  the following commutative diagram of $H$-equivariant maps where the vertical map $\phi_H$ is bijecive: 
\[
\begin{CD}
B^{H} @>{\tau}>> B^{H}   \\
@V{\phi_H}VV @VV{\phi_H}V \\
A^G @>{\sigma_s}>> A^G
\end{CD}
\] 
 
\par 
For (i), suppose that  $\sigma_s$ is injective and $G$ is a surjunctive group.  As mentioned in the beginning of the proof, $\sigma_s$ is stably injective  as it is injective and $\Sigma(s)=Gs$. By \cite[Theorem~A]{phung-tcs}, we deduce that $\sigma_s$ is  reversible. It follows from the diagram that the CA $\tau$ is  injective. As a subgroup of a surjunctive group, $H$ is also  
surjunctive. Therefore,   $\tau$ must be surjective and we deduce again from the diagram  that  $\sigma_s$ is  surjective. As a bijective and reversible NUCA with finite memory,  $\sigma_s$ is invertible (and in fact stably invertible by Lemma~\ref{l:equivalent-stable-invertible}). The proof of (i) is thus complete. 
\par 
For (ii), suppose that the group $G$ is post-injunctive and $\sigma_s$ is post-surjective. It follows at once from the above diagram and the definition of post-surjectivity that $\tau$ is also post-surjective. We remark that subgroups of a post-injunctive are also post-injunctive. Therefore, $H$ is post-injunctive and thus $\tau$ must be pre-injective. 
We deduce that $\tau$ is  invertible (see e.g. \cite[Theorem~13.4]{phung-tcs}) and in particular, it is injective. It follows from the above diagram   that $\sigma_s$ is also injective. As $\Sigma(s)=Gs$, we deduce  that $\sigma_s$ is stably injective and thus reversible. As $\sigma_s$ is post-surjective, it is also surjective. We can conclude that $\sigma_s$ is invertible.  The proof of Theorem~\ref{t:finite-orbit} is complete. 
\qed 

\section{A more constructive proof of Theorem~\ref{t:main-C} for finitely generated sofic group universes} 
\label{s:another-proof}

Our main purpose in this expository section is to present a more self-contained and constructive proof of Theorem~\ref{t:main-C} in the case when the  universe $G$ is a finitely generated sofic group. As mentioned in the Introduction, sofic groups remain the only exclusive known examples of groups for which Gottschalk's surjunctivity conjecture and its dual conjecture hold true. Our proof of Theorem~\ref{t:main-C} for local perturbations of CA over sofic group universes is based on the Cayley graph representation of  finitely generated sofic groups. As such, our proof generalizes the invertibility proof obtained in  \cite{kari-post-surjective} for post-surjective CA also over sofic groups. We will first recall the definition and basic properties of sofic groups in Section~\ref{s:sofic-groups} to prepare for the proof of Theorem~\ref{t:main-C} given in Section~\ref{s:another-proof-B-sofic}. Note that by a similar proof strategy, we can also obtain a more constructive proof of Theorem~\ref{t:main-A} in the case of finitely generated sofic group universes.  

\subsection{Preliminaries on finitely generated sofic groups}
\label{s:sofic-groups}

Let $\Delta$ be a finite set. We define a   $\Delta$-labeled graph as a pair $\GG= (V,E)$  
where $V$ is the set of {vertices} 
and $E \subset V \times \Delta \times V$ is the set of $\Delta$-labeled edges. 
The length of a path $\rho$ in $\GG$ is denoted by $l(\rho)$.  
For $v, w\in V$, we set   $d_\GG(v,w)= \min \{ l(\rho): \text{$\rho$ is a path from $v$ to $w$}\}$ if  $v$ and $w$  are connected by a path, and 
$d_\GG(v,w)= \infty$ 
otherwise. 
For $v\in V$ and $r \geq 0$, the ball $B_\GG(v,r)$ of radius $r$ centered at $v$  in $\GG$ 
\[
B_\GG(v,r)= \{w\in V: d_\GG (v,w) \leq r\}
\]
is a $\Delta$-labeled subgraph of $\GG$. 
Let $(V,E)$ and $(W,F)$ be $\Delta$-labeled graphs. 
A map $\phi \colon V \to W$ satisfying   $(\phi(v),s, \phi(w)) \in F$ for all $(v,s,w)\in E$ is called  a $\Delta$-labeled graph homomorphism 
from $(V, E)$ to $(W,F)$. 
A bijective $\Delta$-labeled graph homomorphism $\phi \colon V \to W$ is a 
$\Delta$-labeled graph isomorphism if its inverse  $\phi^{-1}\colon W \to V$ 
is a $\Delta$-labeled graph homomorphism. 
\par
Let $G$ be a finitely generated group and let $\Delta\subset G$ 
be a finite symmetric generating subset, i.e., $\Delta=\Delta^{-1}$. 
The Cayley graph of $G$ with respect to $\Delta$ is the connected $\Delta$-labeled graph $C_\Delta(G) = (V,E)$ 
where $V = G$ and $E=\{(g,s,gs): g\in G \text{ and } s \in \Delta)\}$.  
For $r\geq 0$, we denote 
\[
B_\Delta(r)= B_{C_\Delta(G)}(1_G,r).   
\]  
\par 
\noindent
We have the following characterization of finitely generated sofic groups in terms of Cayley graphs due to Weiss (\cite{weiss-sgds}, see also \cite[Theorem~7.7.1]{csc-book}). 

\begin{theorem}
\label{t:sofic-character}
Let $G$ be a finitely generated group. 
Let $\Delta\subset G$ be a finite symmetric generating subset. 
Then the following are equivalent: 
\begin{enumerate} [\rm (a)]
\item
the group $G$ is sofic;
\item
for all $r, \varepsilon >0$, there exists a finite $\Delta$-labeled graph $\GG=(V,E) $ 
satisfying 
\begin{equation*} 
\vert V(r) \vert \geq (1 - \varepsilon) \vert V \vert,
\end{equation*}
where $V(r)\subset V$ consists of $v \in V$ such that there exists a unique $S$-labeled graph
isomorphism $\psi_{v,r} \colon  B_\GG(v,r)\to B_\Delta(r)$ with $\psi_{v,r}(v) = 1_G$. 
\end{enumerate}
\end{theorem}
\noindent
Note that if $0 \leq r\leq s$ then $V(s) \subset V(r)$ since every $\Delta$-labeled graph isomorphism $\psi_{v,s} \colon  B_\GG(v,s) \to B_\Delta(s)$ induces by restriction  
a $\Delta$-labeled graph isomorphism $ B_\GG(v,r) \to B_\Delta(r)$. 
We will also need the  following Packing lemma  (\cite{weiss-sgds}, see also \cite[Lemma~7.7.2]{csc-book}).  

\begin{lemma}
\label{l:sofic-B-V}
With the notation as in Theorem~\ref{t:sofic-character}, the following hold 
\begin{enumerate}[\rm(i)]
\item
$B_\GG(v,r) \subset V(kr)$ for all $v \in V((k+1)r)$ and $k \geq 0$;
\item
There exists a finite subset $W \subset V(3r)$ such that the balls $B_\GG(w,r)$ 
are pairwise disjoint for all $w\in W$ and that 
$V(3r) \subset \bigcup_{w\in W} B_\GG(w,2r)$. 
\end{enumerate}
\end{lemma}

\subsection{Another proof of Theorem~\ref{t:main-C} for sofic group universes}  

\label{s:another-proof-B-sofic} 
To recall the notations, let $G$ be a finitely generated sofic group and let $A$ be a finite alphabet. Let $M\subset G$ be a finite subset and $S=A^{A^M}$. Let $s \in S^G$ be  asymptotically constant such that $\sigma_s$ is stably post-surjective. We will first prove that $\sigma_s$ is pre-injective.  The theorem is trivial when $\vert A \vert \leq 1$ so we only need to consider the case  $\vert A \vert \geq 2$.   
\par  
Let $\Delta$ be a symmetric generating set of $G$, i.e., $\Delta=\Delta^{-1}$, such that $1_G \in \Delta$. 
By hypothesis, we can find a constant configuration $c \in S^G$ and a finite subset $F\subset G$   such that $s\vert_{G \setminus F} =  c\vert_{G \setminus F}$.   Since $\sigma_s$ is stably post-surjective, we infer from  \cite[Lemma~13.2]{phung-tcs}  a finite subset $E \subset G$  such that for all $g \in G$ and  $x, y\in A^G$ with $y\vert_{G \setminus \{g\}} =\sigma_s(x)\vert_{G \setminus \{g\}}$, there exists $z \in A^G$ such that $\sigma_s(z)=y$ and 
$z\vert_{G \setminus gE}= x\vert_{G \setminus gE}$. 
Suppose on the contrary that $\sigma_s$ is not  
pre-injective. Then  there exist distinct asymptotic configurations $p, q \in A^G$ with  $\sigma_s(p)= \sigma_s(q)$. 
Up to enlarging $M$ and $F$, we can suppose that  $p\vert_{FM} \neq q\vert_{FM}$. 
We fix $r \in \N$ such that $r \geq 1$ and   
$B_\Delta(r) \supset M \cup E\cup F$ where $B_\Delta(r) \subset G$ denotes the ball of radius $r$ centered at $1_G$   in the Cayley graph $C_\Delta(G)$ of $G$.  
Up to enlarging $M$, $E$, and $F$ again if necessary,  we can suppose without loss of generality  that 
$$M=E=F=B_\Delta(r).$$  
We infer from \eqref{e:induced-local-maps},  \eqref{e:induced-local-maps-general}, and the relation $\sigma_s(p)= \sigma_s(q)$  
that: 
\begin{align} 
\label{e:mutually-erasable-pattern}
   f_{F,s\vert_F}^{+M}(p\vert_{FM}) = f_{F,s\vert_F}^{+M}(q\vert_{FM}).
\end{align}
\par 
\noindent 
Let $R=4r$ and let  $\varepsilon \in (0, \frac{1}{2})$ be small enough
  so that 
  \begin{equation}
  \label{e:post-surjective-var-epsilon-3-1}
    \vert A \vert^{\varepsilon} \left( 1 - \vert A \vert^{-|B_\Delta(R)|} \right)^{\frac{1}{2|B_\Delta(2R)|}}
    < 1,
  \end{equation}
\par 
\noindent 
By Theorem~\ref{t:sofic-character}, there exists a finite $\Delta$-labeled graph $\GG=(V,E)$ 
 such that
\begin{equation} 
\label{e:sofic-V-2post-surjective}
    \vert V(3R) \vert \geq (1 - \varepsilon) \vert V \vert \geq (1 - \varepsilon)\vert V(R) \vert
\end{equation}
where
  $V(n)$ consists   of $v \in V$ such that there exists a $\Delta$-labeled graph
isomorphism $\psi_{v,n}\colon  B_{\GG}(v,n) \to B_\Delta(n)$ mapping 
 $v$ to $1_G$. 
 \par 
In what follows, we  construct  an auxiliary surjective  map $\Phi \colon A^{V(R)} \to A^{V(3R)}$ which captures various  induced local maps of $\sigma_s$ via the the isomorphisms $\psi_{v,r}$ where  $v \in V(3R)$. 
By Lemma~\ref{l:sofic-B-V}.(ii), there exists a subset $W \subset V(3R) $ such that the balls 
$B_\GG(w,R)$, for all $w\in W$,  are pairwise disjoint  and $V(3R) \subset \bigcup_{w \in W}B_\GG(w,2R)$. In particular,  
\begin{equation}
\label{e:v'-and-v-post-surjective-3}
    \vert V(3R) \vert \leq \vert W \vert \vert B_\Delta(2R) \vert.
\end{equation}
\noindent 
Let 
$\overline{W} = \coprod_{w \in W} B(w,R) \subset V(2R)$. 
For $x \in A^{V(R)}$ and $v \in V(3R)$, we define  
\begin{align}
\label{eq:Phi-map-dual}
\Phi(x)(v) 
=
\begin{cases}
    s(\psi^{-1}_{w,R}(v)) (x\vert_{B(v,r)}\circ \psi_{v,r}) & \text{ if } v\in B(w,R) \text{  for some } w \in W\\
    c(1_G)\left(x\vert_{B(v,r)} \circ \psi_{v,r}\right) & \text{ if } v \in V(3R) \setminus \overline{W}.  
\end{cases}
\end{align}
\par 
\noindent
For $v\in B(w,R)\cap V(3R)$ where $w \in W$, we have for  $g= \psi^{-1}_{w,R}(v) \in B_\Delta(R)$ and for all  $x \in A^{V(R)}$ that: 
\begin{align}
\label{eq:Phi-equal-sigma-s-locally-1-dual}
    \Phi(x)\vert_{B(v, r)} \circ \psi_{v,r} \circ \gamma_{g,B_\Delta (r)} = f_{gB_\Delta (r), s\vert_{ g B_\Delta (r)}}^{+B_\Delta (r)}(x\vert_{B(v,2r)} \circ \psi_{v,2r}\circ \gamma_{g,B_\Delta (2r)}).
\end{align}
 
\par 
\noindent
For $v\in V(3R) \setminus \overline{W}$ and  $x \in A^{V(R)}$, observe that     
\begin{align}
\label{eq:Phi-equal-sigma-s-locally-2-dual}
   \Phi(x)\vert_{B(v, r)} \circ \psi_{v,r} =  f_{B_\Delta (r), s\vert_{B_\Delta (r)}}^{+B_\Delta (r)}(x\vert_{B(v,2r)} \circ \psi_{v,2r}). 
\end{align}
 
\par 
\noindent
Let $y \in A^{V(3R)}$ and consider an arbitrary pattern $x \in A^{V(R)}$. Let $v_1,..., v_n$ be elements of $V(3R)$. 
We describe inductively below a procedure on $1\leq i\leq n$ which allows us to obtain a sequence $x_0=x, x_1,  ..., x_n$ such that
 $$\Phi(x_i)\vert_{\{v_1, ..., v_i\}}=y\vert_{\{ v_1, ..., v_i\}}, \quad  i=1,..., n.$$ 
 For $i=1,..., n$, let $g_i= \psi^{-1}_{w,R}(v_i)$ if $v_i \in B(w,R)\cap V(3R)$ for some $w \in W$ and $g_i= 1_G$ otherwise. 
By \eqref{eq:Phi-equal-sigma-s-locally-1-dual}, \eqref{eq:Phi-equal-sigma-s-locally-2-dual}, and the choice of $E$,   we can modify, via  the isomorphism $\psi_{v, r}\circ \gamma_{g_i, B_\Delta(r)}$,  the restriction  $x_{i-1}\vert_{B(v_{i}, r)}$ to obtain a new configuration $x_{i}\in A^{V(R)}$ such that  $x_{i}\vert_{V(R)\setminus B(v_{i}, r)}= x_{i-1}\vert_{V(R)\setminus B(v_{i}, r)}$ and 
$$\Phi(x_{i})\vert_{V(3R)\setminus \{v_{i}\}}= \Phi(x_{i-1})\vert_{V(3R)\setminus \{v_{i}\}}, \quad \quad \Phi(x_{i})(v_{i})= y(v_{i}). 
$$
Consequently, $\Phi(x_n)= y$ and we deduce that   $\Phi$ is surjective. In particular,    
  \begin{equation} 
  \label{e:surj-post-phi-surjective-v-3r}
  \vert \Phi(A^{V(R)})\vert = \vert A^{V(3R)}\vert .
  \end{equation}
  \par 
  \noindent 
By  \eqref{e:mutually-erasable-pattern},  we have    
$f_{B_\Delta(r),s\vert_{B_\Delta(r)}}^{+B_\Delta (r)}(p\vert_{B_\Delta(2r)}) = f_{B_\Delta(r),s\vert_{B_\Delta(r)}}^{+B_\Delta (r)}(q\vert_{B_\Delta(2r)})$. Let 
$$\Lambda = \prod_{w \in W}\psi_{w,R}^{-1}(\{p\vert_{B_\Delta(R)}, q\vert_{B_\Delta(R)}\}) \subset A^{\overline{W}}.$$ 
Then for 
every $e \in A^{V(R) \setminus {\overline{W}}}$ and all $x,y \in \Lambda$, we deduce from \eqref{eq:Phi-map-dual} and $R=4r$ that $ 
    \Phi(e \times x) =  \Phi(e \times y)$. 
Consequently,  
\[
\Phi(A^{V(R)}) = \Phi\left(A^{V(R) \setminus {\overline{W}}}\times \prod_{w \in W} \left( A^{B(w,R)}\setminus \psi_{w,R}^{-1}(q\vert_{B_\Delta(R)})\right)\right). 
\]
Therefore,  as $|B(w,R)|=|B_\Delta(R)|$ for all $w \in W$, we obtain the estimation 
\begin{align}
\label{e:surj-post-phi-surjective-v-3r-c}
     \vert \Phi(A^{V(R)})\vert 
     & \leq \left|A^{V(R) \setminus {\overline{W}}}\times \prod_{w \in W} \left( A^{B(w,R)}\setminus \psi_{w,R}^{-1}(q\vert_{B_\Delta(R)})\right)\right|\\
     & \leq \vert A \vert^{\vert V(R) \setminus {\overline{W}}\vert} (\vert A \vert^{\vert B_\Delta(R)\vert}-1)^{\vert W\vert}. \nonumber 
\end{align}
Finally, we find that 
  \begin{align*}
    \vert A\vert^{\vert {V(3R)}\vert } &  =  \vert \Phi(A^{V(R)})\vert  & \text{(by } \eqref{e:surj-post-phi-surjective-v-3r}) \\ 
    & \leq 
   \vert A \vert^{\vert V(R) \vert - \vert{\overline{W}}\vert} (\vert A \vert^{\vert B_\Delta(R)\vert}-1)^{\vert W\vert}  
    & \text{(by } \eqref{e:surj-post-phi-surjective-v-3r-c}) 
    \\
    & = 
   \vert A \vert^{\vert V(R) \vert} (1 - \vert A \vert^{- \vert B_\Delta(R)\vert})^{\vert W\vert}  
    &  \text{(as } |\overline{W}|=|B_\Delta(R)| |W|) 
    \\ 
    & \leq 
     \vert A \vert^{\vert V(R) \vert} (1 - \vert A \vert^{- \vert B_\Delta(R)\vert})^{\frac{\vert V(3R)\vert}{  \vert B_\Delta(2R) \vert}} & \text{(by }\eqref{e:v'-and-v-post-surjective-3}) \\
      &  \leq 
     \vert A \vert^{\vert V(R) \vert} (1 - \vert A \vert^{- \vert B_\Delta(R)\vert})^{\frac{\vert V(R)\vert}{2  \vert B_\Delta(2R) \vert}}
      & (\text{by }\eqref{e:sofic-V-2post-surjective})
      \\
      &  < 
     \vert A \vert^{\vert V(R) \vert}  \vert A \vert^{- \varepsilon \vert V(R) \vert}
      & (\text{by }\eqref{e:post-surjective-var-epsilon-3-1}) \\
      & \leq \vert A \vert^{\vert V(3R)\vert} & \text{(by } \eqref{e:sofic-V-2post-surjective}) 
\end{align*} 
  which is a contradiction. Hence, $\sigma_s$ must be pre-injective. Since $\sigma_s$ is also stably post-surjective by hypothesis, we conclude from \cite[Theorem~13.4]{phung-tcs} that  $\sigma_s$ is invertible. By Lemma~\ref{l:equivalent-stable-invertible},   $\sigma_s$ is  stably invertible and the proof of Theorem~\ref{t:main-C} for finitely generated group universes is complete. \qed 

\subsubsection{\discintname}
 The authors have no competing interests to declare that are
relevant to the content of this article. 

 \bibliographystyle{splncs04}

\end{document}